\documentclass[12pt, reqno]{amsart}

\usepackage{cite}

\usepackage{color}
 \usepackage{amscd}

\usepackage{amsrefs}

 \usepackage{amssymb, amsmath}

 \input xy
  \xyoption{all}
 
 \input diagxy

\usepackage{tikz}
\usepackage{tikz-3dplot}

\usetikzlibrary{decorations.markings}

\usetikzlibrary{shapes,snakes}

\tikzstyle{mybox} = [draw=black,   
    rectangle, rounded corners, inner sep=4pt, inner ysep=2pt]

\tikzset{middlearrow/.style={
        decoration={markings,
            mark= at position 0.5 with {\arrow{#1}} ,
        },
        postaction={decorate}
    }
}

\newcommand{\tikzAngleOfLine}{\tikz@AngleOfLine}
  \def\tikz@AngleOfLine(#1)(#2)#3{%
  \pgfmathanglebetweenpoints{%
    \pgfpointanchor{#1}{center}}{%
    \pgfpointanchor{#2}{center}}
  \pgfmathsetmacro{#3}{\pgfmathresult}%
  }

\usepackage{epsf}
\usepackage{epsfig}

    \newtheorem{prop}{Proposition}

  \numberwithin{equation}{section}
 \numberwithin{prop}{section}

\newcommand{\Beq}{\begin{equation}}
\newcommand{\Eeq}{\end{equation}}
\newcommand{\Beqr}{\begin{eqnarray}}
\newcommand{\Eeqr}{\end{eqnarray}}

\newcommand{\mbt}{{\mathbb  T}}

\newcommand{{\mbbx}}{{\mathbb  X}}

\newcommand{{\mbx}}{{\mathbf  X}}

\newcommand{\mbr}{{\mathbb R}}

\newcommand{\mca}{{\mathcal A }}

\newcommand{\mci}{{\mathcal I }}

\newcommand{\ovGb}{  {\overline{\mathbf G}}             }

\newcommand{\ovg}{  {\overline g}}  
\newcommand{\ovG}{  {\overline G}}      
\newcommand{\ovGt}{  {\overline {\mathbf G}}_{\tau}}     

\newcommand{\ovgam}{  {\overline \gamma}}

\newcommand{\ovl}{\overline }

\newcommand{\ovph}{  {\overline \phi}}  

\newcommand{\ovt}{{\overline  \theta}}

\newcommand{{\wtlg}}{\widetilde\gamma }

\newcommand{{\wtlG}}{\widetilde\Gamma }

\newcommand{{\wtlv}}{\widetilde v }

\newcommand{{\tlg}}{\tilde\gamma }

\newcommand{{\tlG}}{\tilde\Gamma }

\newcommand{\mbbc}{{\mathbf C}}

\newcommand{\mbbb}{\mathbf {B} }

\newcommand{\mbg}{\mathbf {G} }
\newcommand{\mbh}{\mathbf {H} }

\newcommand{\mbp}{\mathbf {P} }
\newcommand{\mbu}{\mathbf {U} }

\newcommand{\Obj}{{\rm Obj }}
\newcommand{\Mor}{{\rm Mor }}

\begin{document}
\title{Construction of Categorical Bundles from Local Data}

\author{Saikat Chatterjee }
\address{Saikat Chatterjee , School of Mathematics, Indian Institute of Science Education and Research\\
CET Campus\\ Thiruvananthapuram, Kerala-695016\\
India}
\email{saikat.chat01@gmail.com}

\author{Amitabha Lahiri}
\address{Amitabha Lahiri, S.~N.~Bose National Centre for Basic Sciences \\ Block JD,
  Sector III, Salt Lake, Kolkata 700098 \\
  West Bengal, India}
  \email{amitabhalahiri@gmail.com}

\author{Ambar N. Sengupta }
\address{Ambar N. Sengupta, Department of Mathematics\\
  Louisiana State University\\  Baton
Rouge, Louisiana 70803, USA}
\email{ambarnsg@gmail.com}



\def\xypic{\hbox{\rm\Xy-pic}}

\begin{abstract}
A categorical principal bundle is a structure comprised of categories that is analogous to a classical principal bundle; examples arise from geometric contexts involving bundles over path spaces. We show how a categorical principal bundle can be constructed from local data specified through transition functors and natural transformations.
\end{abstract}


\maketitle

\section{Introduction}\label{s:int}

A categorical principal bundle $\pi:\mbp\to\mbbb$ is a structure analogous to a classical principal bundle, but with all the spaces involved replaced by categories and maps by functors. Of interest to us is the case where these categories have a geometric significance; for example, there is an underlying classical principal bundle $\pi:P\to B$ and the objects of the `base category'  $\mbbb$ are the  points of   $B$ while the morphisms arise from paths on  $B$. In the  `bundle category' $\mbp$ the objects are the points of    $P$ and morphisms are of the form $(\ovgam, h)$, where $\ovgam$ comes from a path on $P$ that is horizontal with respect to a connection form on $\pi:P\to B$ and $h$ is a `decoration' drawn from a Lie group $H$. There are different notions of local triviality for such structures. In this paper we we construct a categorical bundle from local data. The local data do not come as a traditional cocycle of transition functions but rather as functors that fall short of a cocycle relation. We describe a quotient procedure that leads to   ``functorial cocycles''  and then construct a categorical bundle from such cocycles.

The literature in category theoretic geometry has grown rapidly in recent years.  We mention here the works of  Abbaspour and Wagemann [1],  Attal [3, 4], Baez et al.  [5, 6], Barrett [7], Bartels [8], Breen and Messing [9], Parzygnat [20], Picken et al. [10, 17, 18], Soncini and Zucchini [23], Schreiber and Waldorf [24, 25], Viennot [26] and Wang [27, 28].  Among others, the works [sec. 3, 1]; [Propn 2.2, 5]; [sec. 2, 6] study the relationship between gerbe local data and 2-bundles, as well as connections on such structures. Our framework and structures are closely related in spirit but there are differences. In particular we have two categorical groups $\mbg$ and $\mbh$ in terms of which the local data are specified. The technical nature of what a categorical bundle is also specified differently in our approach.

\subsection{Results and organization} We begin in section \ref{s:bn} with a summary of essential notions and notation concerning categorical groups, categorical bundles, and categories arising from points and paths on manifolds.  All through this paper the base space of the bundle is a manifold $B$, and local data is specified relative to an open covering $\{U_i\}_{i\in \mci}$ of $B$.  Associated to these sets are categories $\mbu_i$ (objects are points of $U_i$ and morphisms arise from paths on $B$ that lie inside $U_i$) and overlap categories such as $\mbu_{ik}=\mbu_i\cap\mbu_k$.  

In section \ref{s:cocy} we work with two categorical groups $\mbg$ and $\mbh$,  and introduce {\em gerbal cocycles} (subsection \ref{ss:gerbcoc}), which are analogous to classical cocycles except that they fall short of satisfying the exact identities needed for classical cocycles. Then  we construct {\em functorial} forms of these gerbal cocycles; briefly put  they are given by functors
\begin{equation}\label{E:theik1}
\theta_{ik}:\mbu_{ik}\to \mbg
\end{equation}
and natural isomorphisms
\begin{equation}\label{E:Theik1}
\mbt_{ikm}:\theta_{ik}\theta_{km}\to\theta_{im}.
\end{equation}  
This is explained in subsection \ref{ss:confunct}. We establish several properties of the natural transformations (\ref{E:theik1}).

In order to obtain genuine cocycles we take quotients to form a category $\ovGb$   and a categorical group $\ovGb_{\tau}$ (these are constructed in subsections  \ref{ss:qG}  and \ref{ss:qGt}). 

 We turn next in section \ref{s:cbcld} to the construction of a globally defined categorical bundle over $\mbbb$ (points forming the base manifold $B$ and morphisms arising from paths on $B$). The classical construction of a principal bundle $\pi:X\to B$ from a cocycle of transition functions may be viewed as the construction of a projective limit from a family of trivial bundles
$$U_i\times G\to U_i.$$  
Using this viewpoint we construct in section \ref{s:cbcld} a categorical principal bundle
$$\mbx\to\mbbb,$$
by sewing together the  trivial categorical bundles $\mbu_i\times\ovGb_{\tau}$ using the functorial transition data in (\ref{E:theik1}) and (\ref{E:Theik1}).  Here we use the formalism of {\em quivers}, explained in subsection \ref{ss:quiv}.
 
Using  techniques similar to those described in this paper it is possible to construct a categorical principal $\mathbf G$-bundle (instead of ${\bar {\mathbf G}}_{\tau}$-bundle), which enjoys a {\em weak} local trivialization property, where  local triviality is understood in terms of equivalences rather than isomorphisms. For this it suffices that the target maps
$\tau':J\to H$ and $\tau:H\to G$   satisfy the condition that $ {\rm Ker}(\tau)\cap{\rm Im}(\tauÕ)$  is trivial, where the notation is as explained at the beginning of section \ref{s:cocy}. We will not explore this line of investigation in the present paper.

\section{Basic notions}\label{s:bn}

In this section we summarize the essentials of terminology and notions that we use. The categories we work with are all small categories, the objects and morphisms forming sets. In fact the object sets of the categories we work with are smooth manifolds and functors are, at the level of objects, given by smooth functions.

By a {\em categorical group} ${\mbg}$ we mean a small category  along with a functor
$$\mbg\times\mbg\to\mbg$$
that makes both the object set $\Obj(\mbg)$ and the morphism set $\Mor(\mbg)$ groups. The source and target maps
$$s, t:\Mor(\mbg)\to\Obj(\mbg)$$
are homomorphisms.  We say that the categorical group ${\mbg}$ is a categorical Lie group if $\Obj(\mbg)$ and $\Mor(\mbg)$ are Lie groups and $s$ and $t$ are smooth mappings. Associated to a categorical group $\mbg$ is a {\em crossed module} $(G, H,\alpha,\tau)$, where $G$ and $H$ are groups, and
\begin{equation}\begin{split}
\tau:H\to G & \qquad \hbox{and}\qquad \alpha: G\to {\rm Aut}(H):g\mapsto\alpha_g\end{split}\end{equation}
are homomorphisms satisfying the 
Peiffer identities 
\begin{equation}\label{E:Peiffer}
\begin{split}
\tau\bigl(\alpha_g(h)\bigr) &= g   \tau(h)  g^{-1}\\
\alpha_{\tau(h)}(h') &=hh'h^{-1}
\end{split}
\end{equation}
for all $g\in G$ and $h\in H$. The relationship between ${\mbg}$ and the crossed module is given by
$$G=\Obj(\mbg) \quad\hbox{and}\quad H=\ker s\subset\Mor(\mbg).$$
The morphism group $\Mor(\mbg)$ can be identified with the semidirect product $H\rtimes_{\alpha}G$:
$$\Mor(\mbg)\simeq H\rtimes_{\alpha}G,$$
with  $(h,g)\in H\rtimes_{\alpha}G$ having source $g$ and target $\tau(h)g$:
\begin{equation}\label{E:HGalptaust}
s(h,g)=g\qquad\hbox{and}\qquad t(h,g)=\tau(h)g.
\end{equation}
Composition of morphisms is given in $H\rtimes_{\alpha}G$ by
\begin{equation}\label{E:morcompHG}
(h_2,g_2)\circ (h_1, g_1)=(h_2h_1, g_1),
\end{equation}
in contrast to the product operation in $\Mor(\mbg)$ which is given by the semidirect product operation
\begin{equation}\label{E:semidprod}
(h_2,g_2)(h_1,g_1)= \bigl(h_2\alpha_{g_2}(h_1), g_2g_1\bigr).
\end{equation}
The categorical group $\mbg$ is a Lie group if and only if $G$ and $H$ are Lie groups and the mappings $\tau:h\mapsto\tau(h)$ and $(h,g)\mapsto\alpha_g(h)$ are smooth. 

It will often, but not always, be convenient to identify $H$ and $G$ with the subgroups $H\times\{e\}$ and $\{e\}\times G$ in $H\rtimes_{\alpha}G$, so that $(h,g)$ can be written simply as a product:
\begin{equation}\label{E:prodhg}
hg=(h,g).
\end{equation}
 As a  consequence of the first Peiffer identity  the image $\tau(H)$ is a normal subgroup of $G$:
\begin{equation}\label{E:normaltau}
\hbox{$g\tau(h) g^{-1}= \tau\bigl(\alpha_g(h)\bigr)\in\tau(H)$
   for all $h\in H$ and $g\in G$.}
   \end{equation}

   By a {\em categorical principal bundle} with {\em structure categorical group} $\mbg$ we mean a functor 
   $$\pi: \mbp \to\mbbb $$
   that is surjective  both on the level of objects and on the level of morphisms, along a functor
      $$\mbp\times\mbg\to\mbp$$
      that is a free right action   both on objects and on morphisms, such that $\pi(pg)=\pi(p)$ for all objects/morphisms $p$ of $\mbp$ and all objects/morphisms $g$ of $\mbg$. (For more on categorical principal bundles we refer to [12].) This is a `bare bones' definition; in practice we are only concerned with those examples in which  $\mbg$ is a categorical Lie group, $\Obj(\mbp)$ and $\Obj(\mbbb)$ are smooth manifolds,  and the object bundle
      $$\Obj(\mbp)\to\Obj(\mbbb)$$
      is a principal $G$-bundle, where $G=\Obj(\mbg)$. The morphisms of $\mbbb$ arise from paths on $B=\Obj(\mbbb)$, as we now discuss.

       \subsection{Categories from points and paths}\label{ss:cpp}
  
  We turn now to categories of points and paths.
  Associated to a smooth manifold $M$ there is a category ${\mathbf M}$ whose   objects are the points of $M$, and whose  morphisms are all piecewise smooth paths on $M$. Let us specify this in more detail. The paths we use are smooth mappings of the form $[a,b]\to M$, where $a,b\in\mbr$ with $a<b$, and the paths are assumed to be constant near the initial time $a$ and  the terminal time $b$. Paths $\gamma_1:[a,b]\to M$ and $\gamma_2:[c,d]\to M$ are identified if there is a constant $r$ such that $[c,d]=[a,b]+r$ and 
\begin{equation}\label{E:gam2gam1}
\gamma_2(t) =\gamma_1(t-r)\qquad\hbox{for all $t\in [c,d]$.}
\end{equation}
  The source of $\gamma$ is the initial point and the target of $\gamma$ is the terminating point; often we will find it notationally convenient to write $\gamma_0$ or even $\gamma(0)$ to denote the source $s(\gamma)$, and $\gamma_1$ or $\gamma(1)$ to denote the target $t(\gamma)$.
Composition of morphisms is defined by composition of paths; the requirement that paths be constant near their initial and terminal times ensures that the composition of two such paths is smooth and has the same property.

\begin{figure}[h]

 \tdplotsetmaincoords{70}{110}
\begin{tikzpicture}[scale=.9,tdplot_main_coords]
   
    \def\x{.5}
   
    \filldraw[
        draw=purple,%
        fill=purple!20,%
    ]          (-1,8,0)
            -- (7,8,0)--(7,17,0)
 -- (-1,17,0)--(-1,8,0);


                \fill [blue] (6,12,1) circle[radius=3pt] node[anchor=east] {$\gamma_0=\gamma(a)$};
                
                 \fill [blue] (6,16,1) circle[radius=3pt] node[anchor=west] {$\gamma_1=\gamma(b)$};

  \coordinate [label=above: {$\gamma$}] (gam) at (5,14,1.3);

                \coordinate [ label=above: {$M$}] (B) at (6,14,0);

                   \draw[middlearrow={>}, very thick, purple] (6,12,1) to [out=90,in=185]  (5,14,2) to [out=0, in =180] (6,16,1);
                   
                   \node [mybox, right=30,below=6] (box) at (5,22,3){%
    \begin{minipage}{0.35\textwidth}
 A morphism $\gamma_0\to\gamma_1$ of $\mathbf M$ arises from a path $\gamma:[a,b]\to M$
    \end{minipage}
};

\end{tikzpicture}

\caption{The category ${\mathbf M}$}
    \label{F:decbun}
    \end{figure}
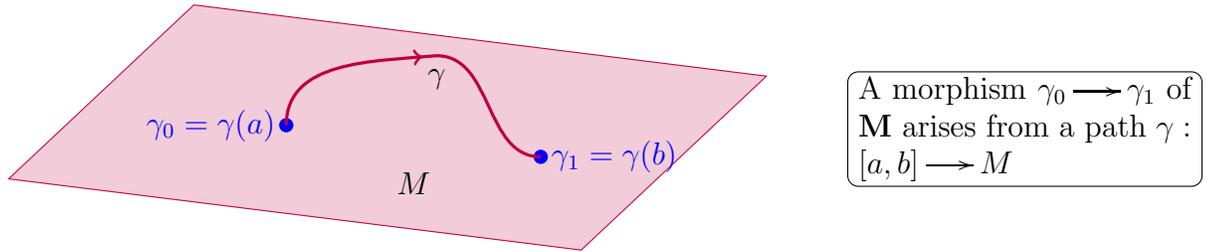

\subsection{Triviality and local triviality} There are different notions of triviality that are of interest for categorical bundles.  The product categorical bundle over a base category $\mbu$ and having structure categorical group $\mbg$  is given by the  projection functor
$$\mbu\times\mbg\to\mbu$$
and the obvious right action of $\mbg$ on $\mbu\times\mbg$. The simplest and strongest notion of triviality of a categorical principal $\mbg$-bundle $\mbp\to\mbu$  to require that there be an isomorphism of categories $\Phi: \mbp\to \mbu\times \mbg$, of appropriate smoothness, that respects the action of $\mbg$ as well as the projection functor. An alternate notion, which allows a richer geometric structure, is explored in  [15]. An even weaker notion is to require that $\Phi$ be an equivalence, rather than an isomorphism. Corresponding to these notions of global triviality there are notions of local triviality for categorical bundles.

\subsection{Covering subcategories}\label{ss:covsubc} We will find it convenient to work with subcategories of $\mbbb$ for which paths initiate and terminate in specified sets of the open covering $\{U_i\}_{i\in {\mathcal I}}$ of a manifold $B$. To this end let $\mbu_{i}$ be the category whose object set is
 $U_i$
and whose morphisms are the morphisms  $\gamma$ of $\mbbb$ that lie entirely inside $U_i$.  
  
  It is also useful to introduce  
the overlap category
\begin{equation}\label{D:objUijkl}
\mbu_{ik} =\mbu_i\cap\mbu_k,
\end{equation}
the object set and morphism sets being just the intersections of the corresponding sets for $\mbu_i$ and $\mbu_k$.  Of course, this overlap category is defined only if $U_i\cap U_k\neq\emptyset$. 
  Analogously, we also have triple overlap categories $\mbu_{ikm}$, if $U_j\cap U_k\cap U_m$ is nonempty, and, more generally, $\mbu_I$, for any finite subset $I\subset \mci$ for which the intersection
    \begin{equation}\label{E:UIdef}
    U_I\stackrel{\rm def}{=}\cap_{i\in I}U_i
    \end{equation}
    is nonempty. We denote by $S_{\mci}$ the set of all nonempty subsets $I$ of $\mci$:
    \begin{equation}\label{E:defSI}
    S_{\mci}=\{I\subset \mci\,:\, U_I\neq\emptyset\}.
    \end{equation}
    If $I, J\in S_{\mci}$ with $I\subset J$ then  $U_J\subset U_I$ and so we have a functor
    \begin{equation}
    \mbu_{J}\to\mbu_I
    \end{equation}
    induced by the inclusion $U_J\to U_I$.  Thus, if we denote by ${\mathbf S}_{\mci}$ the category whose object set is $S_{\mci}$ and whose morphisms $I\to  I$ are the inclusion maps $J\to I$, then
 \begin{equation}\label{E:ICMan}
 \mbu_{\cdot}: I\mapsto \mbu_I
     \end{equation}
    specifies a functor from the category ${\mathbf S}_{\mci}$ to the category ${\bf CMan}$ whose objects are categories ${\mathbf M}$ arising from manifolds and whose morphisms arise from smooth mappings between manifolds.   The functor $\mbu_{\cdot}$ is specified on morphisms in the obvious way: it carries the morphism $I\to J$ in ${\mathbf S}_{\mci}$ to the `inclusion' functor $\mbu_I\to\mbu_J$. 
    
    The set-theoretic union of the morphism sets $\Mor(\mbu_i)$ is not generally equal to $\Mor(\mbbb)$, and so the categories $\mbu_i$ do not `cover' $\mbbb$ in any literal sense. However, we do have the functors
    $${\rm inc}_I: \mbu_I\to \mbbb$$
    induced by the inclusion maps $U_I\to B$.  If a morphism $\gamma_i$ in $U_i$ and a morphism $\gamma_j$ in $U_j$ happen to arise from the same path, lying inside $U_i\cap U_j$, then  there is a morphism $\gamma_{ij}\in\Mor(\mbu_{ij})$ (arising from the same path again, but viewed as lying in $U_i\cap U_j$), that is carried to $\gamma_i\in\Mor(\mbu_i)$ and to $\gamma_j\in\Mor(\mbu_j)$ by the `inclusion' functors $\mbu_{\{i,j\}}\to\mbu_i$ and  $\mbu_{\{i,j\}}\to\mbu_j$. Then $\mbbb$ along with the `inclusion' functors $\mbu_I\to\mbbb$ is a co-limit for the functor $\mbu_{\cdot}$ given in (\ref{E:ICMan}). 
    
    We present this formalism mainly to motivate the thinking behind the method we use later in section \ref{s:cbcld} to construct a global categorical bundle from local trivial ones.

\section{Cocycles: from gerbal to functorial}\label{s:cocy}

In this section we construct a cocycle with values in a categorical group, starting with some group-valued locally-defined functions that need not  form a cocycle.

All through this paper we will work with a categorical Lie group  $\mbg$, associated with a Lie crossed module  $(G, H, \alpha, \tau)$, and a categorical Lie group $\mbh$, with associated  Lie crossed module  $(H, J, \alpha',\tau')$.

\subsection{Gerbal cocycles}\label{ss:gerbcoc}

We work with a manifold $B$,  and  an open covering
$$\{U_i\}_{i\in {\mathcal I}}.$$
By a {\em gerbal cocycle} associated with this covering and the Lie crossed module $(H, J,\alpha',\tau')$ we mean a collection of smooth functions
  \begin{equation}\label{E:hjgerb}
  h_{ik}:U_i\cap U_k\to H\qquad\hbox{and}\qquad j_{ikm}: U_i\cap U_k\cap U_m\to J,
  \end{equation}
  with $h_{ik}$ defined when $U_i\cap U_k\neq\emptyset$ and $j_{ikm}$ defined when $U_i\cap U_k\cap U_m\neq\emptyset$, such that
  \begin{equation}\label{E:hilikl}
  h_{im}(u) =\tau'\bigl(j_{ikm}(u)\bigr)h_{ik}(u)h_{km}(u)\qquad\hbox{for all $u\in U_{i}\cap U_k\cap U_m$.}
  \end{equation}
  The pattern here is that on the right the effect of $\tau'\bigl(j_{ikm}(u)\bigr)$ is to `combine' the subscripts $ik$ and $km$ into $im$. 
  
   We will use the notation $U_{ik}$ and $U_{ikm}$ for intersections:
  \begin{equation}\label{E:UikUikm}
  U_{ik}=U_i\cap U_k\qquad\hbox{and}\qquad U_{ikm}=U_i\cap U_k\cap U_m.
  \end{equation}

 In terms of the categorical group $\mbh$ we have a morphism
 \begin{equation}\label{E:gerbcocychj}
 \psi_{ikm}(u): h_{ik}(u)h_{km}(u)\to h_{im}(u)\qquad\hbox{for all $u\in U_{ikm}$,}
 \end{equation}
 where 
  $\psi_{ikm}(u)\in\Mor(\mbh)\simeq J\rtimes_{\alpha'}H$ is given by
  $$\psi_{ikm}(u) =\bigl(j_{ikm}(u), h_{ik}(u)h_{km}(u)\bigr).$$

  Now let
  \begin{equation}\label{E:defgikhijk}
  \begin{split}
   {g}_{ik}  &=\tau(h_{ik}):U_{ik}\to G\qquad\hbox{ if $U_{ik}\neq\emptyset$;} \\
  h_{ikm}&=\tau'(j_{ikm}): U_{ikm}\to H\qquad\hbox{ if $U_{ikm}\neq\emptyset$.} 
  \end{split}
  \end{equation}
  Then by (\ref{E:hilikl}) we have
   \begin{equation}\label{E:hijkgijik2}
  \begin{split}
  h_{im}(u) &=h_{ikm}(u)h_{ik}(u)h_{km}(u)\\
  {g}_{im}(u)&=\tau\bigl(h_{ikm}(u)\bigr){g}_{ik}(u){g}_{km}(u)
  \end{split}
  \end{equation}
  for all $u\in U_{ikm}$. Thus the system of functions $\{g_{ik}\}$ and $\{h_{ikm}\}$ is   a gerbal cocycle associated with the covering $\{U_i\}_{i\in {\mathcal I}}$ and the Lie crossed module $(G, H,\alpha,\tau)$.
  
  There are different notational conventions in specifying gerbe data. Ours is consistent with [Definition 7, 1]. 
  
  \subsection{A second gerbe relation} The conditions (\ref{E:hijkgijik2}) imply another relation satisfied by the $h_{ik}$ and $h_{ikm}$ (suppressing the point $u$ for notational ease):
  \begin{equation}\label{E:triplegerbe}\begin{split}
  h_{ijm} \alpha_{g_{ij}}(h_{jkm}) &= h_{ijm} h_{ij}h_{jkm}h_{ij}^{-1}\\& \quad\hbox{(by the second Peiffer identity (\ref{E:Peiffer}))}\\
  &= h_{im}h_{jm}^{-1}h_{ij}^{-1}\cdot h_{ij}\cdot h_{jm}h_{km}^{-1}h_{jk}^{-1}\cdot h_{ij}^{-1}\\
  &\quad\hbox{(using the first relation in (\ref{E:hijkgijik2}))}\\
  &=h_{im}  h_{km}^{-1}h_{jk}^{-1}h_{ij}^{-1}\\
  &=h_{ikm} h_{ik}h_{km} \cdot h_{km}^{-1}h_{jk}^{-1}h_{ij}^{-1}\quad\hbox{(again by  (\ref{E:hijkgijik2}))}\\
  &=h_{ikm} h_{ik}h_{jk}^{-1}h_{ij}^{-1}.  \end{split}
  \end{equation}
  Using the first relation in (\ref{E:hijkgijik2}) once again, we obtain
  \begin{equation}\label{E:gerbe2}
   h_{ijm} \alpha_{g_{ij}}(h_{jkm})=h_{ikm} h_{ijk}.
  \end{equation}
  This along with the second relation in (\ref{E:hijkgijik2}) ensure that the data $\{g_{ij}\}$ and $\{h_{ijk}\}$ are the local data for a gerbe structure  [Definition 7, 1].

 \subsection{Construction of functorial cocycles}\label{ss:confunct} We continue to work with an open covering $\{U_i\}_{i\in\mci}$ of $B$, and the corresponding path categories $\mbu_i$ and overlap categories $\mbu_{ik}$ as discussed in subsection \ref{ss:cpp}.

Starting with a given gerbal cocycle as in (\ref{E:hjgerb}) let us define
\begin{equation}\label{D:tijkl}
\theta_{ik} :{\mbu}_{ik} \to\mbg
\end{equation}
on objects by
\begin{equation}\label{E:tikjlobj}
\theta_{ik}(u) = g_{ik}(u)\stackrel{\rm def}{=}\tau\bigl(h_{ik}(u)\bigr)
\end{equation}
for all $u\in U_{ik}$,  assumed nonempty, and on morphisms  by
\begin{equation}\label{E:tikjlmor}
\begin{split}
\Mor(\mbu_{ik}) &\to \Mor(\mbg)\\
\theta_{ik}(\gamma) = \Bigl(h_{ik}(\gamma), g_{ik}(\gamma_0) \Bigr): g_{ik}(\gamma_0)&\to \tau\bigl(h_{ik}(\gamma)\bigr)g_{ik}(\gamma_0), 
\end{split}
\end{equation}
where $\gamma_0=s(\gamma)$  and
\begin{equation}\label{E:hikgam}
h_{ik}(\gamma) =h_{ik}(\gamma_1)h_{ik}(\gamma_0)^{-1}.\end{equation}
Let us note that in (\ref{E:tikjlmor}) the target of the morphism $\theta_{ik}(\gamma)$ is $g_{ik}(\gamma_1)$:
\begin{equation}\label{E:tartheik}
t\bigl(\theta_{ik}(\gamma)\bigr)=g_{ik}(\gamma_0).
\end{equation} 
In this subsection we show that $\theta_{ik}$ is a functor and  work out some of the signifcant properties of this system of functors.

Let us verify functoriality of $\theta_{ik}$. If $\gamma$ and $\gamma'$ are morphisms in $\mbu_{ik}$ and $\gamma'_0=\gamma_1$ then
\begin{equation}\label{E:gammacomp}
\begin{split}
\theta_{ik}(\gamma'\circ\gamma) &=\bigl(h_{ik}(\gamma'), g_{ik}(\gamma'_0)\bigr)\circ \bigl(h_{ik}(\gamma), g_{ik}(\gamma_0)\bigr)\\
&= \bigl(h_{ik}(\gamma')h_{ik}(\gamma), g_{ik}(\gamma_0)\bigr)\\
&= \bigl(h_{ik}(\gamma'_1)h_{ik}(\gamma_0)^{-1}, g_{ik}(\gamma_0)\bigr)\\
&= \bigl(h_{ik}(\gamma'\circ\gamma), g_{ik}(\gamma_0)\bigr)\\
&=\theta_{ik}(\gamma')\circ \theta_{ik}(\gamma)
\end{split}
\end{equation}
furthermore, $\theta_{ik}$ clearly maps any identity morphism $i_u:u\to u$   in $\mbu_{ik}$ to the  identity morphism 
$$(e, g_{ik}(u)\bigr):g_{ik}(u)\to g_{ik}(u)$$  
in $\Mor(\mbg)$.

If the triple overlap category $\mbu_{ikm}$ is defined then we can restrict the transition functor to obtain a functor
 \begin{equation}\label{E:trest} \theta_{{i\,k\, |m|}}=\theta_{ik}|{\mbu}_{ikm} :{\mbu}_{ikm}  \to\mbg.
\end{equation}
To minimize notational clutter we will drop the subscript $|m|$ and write $\theta_{ik}$ for the restricted functor as well, leaving it to the context to make the intended meaning clear.

Let us now verify that on  $\mbu_{ikm}$  there is a natural transformation
 \begin{equation}\label{E:hjlntt}
 {\mathbb T}_{ikm}: \theta_{ik}\theta_{km}\Rightarrow\theta_{im}
 \end{equation}
 given on any object $u\in U_{ikm}$ by 
 \begin{equation}\label{E:htild}
 \begin{split}
 {\mathbb T}_{ikm}(u) &= \Bigl(h_{ikm}(u), \,{g}_{ik}(u){g}_{km}(u)\Bigr)\in H\rtimes_{\alpha}G\simeq \Mor(\mbg),
 \end{split}
 \end{equation}
 where $h_{ikm}$ is as given in (\ref{E:defgikhijk}). 
 The source of this morphism is the product ${g}_{ik}(u){g}_{km}(u)$ and the target is
 $$t\bigl( {\mathbb T}_{ikm}(u)\bigr) =\tau\bigl(h_{ikm}(u)\bigr){g}_{ik}(u){g}_{km}(u)={g}_{im}(u)=\theta_{im}(u)$$
 by the second equation of the  gerbal relations (\ref{E:hijkgijik2}).
 
 \begin{prop}\label{E:nattransTikm} With notation as above, $\mbt_{ikm}$ is a natural transformation. \end{prop} 
 Thus $\{\theta_{ik}\}$, along with $\{{\mathbb T}_{ikm}\}$,  is a {\em functorial cocycle}, by which we mean a system of functors $\theta_{ik}:\mbu_{ik}\to\mbg$ satisying the relation
 \begin{equation}\label{E:hjlntt2}
 {\mathbb T}_{ikm}: \theta_{ik}\theta_{km}\Rightarrow\theta_{im},
 \end{equation}
 where ${\mathbb T}_{ikm}$ are natural transformations. 
 \begin{proof} The main task is to verify that for any morphism $\gamma:u\to v$ in $ \mbu_{ikm}$ the diagram

\begin{equation} \label{E:Tnatural}
\xymatrix{
         \ar[d]_{ \theta_{ik}(\gamma)\theta_{km}(\gamma) }  \theta_{ik}(u)\theta_{km}(u)       \ar[r]^-{\mbt_{ikm}(u)} &  \theta_{im}(u) \ar[d]^{\theta_{im}(\gamma)} \\
\theta_{ik}(v)\theta_{km}(v)  \ar[r]_-{\mbt_{ikm}(v)}& \theta_{im}(v) 
}
\end{equation}
commutes.  To this end let us first  compute  $ \mbt_{ikm}(v)\circ \bigl(\theta_{ik}(\gamma)\theta_{km}(\gamma)\bigr)$. 

Let us recall the way composition works in $\Mor(\mbg)\simeq H\rtimes_{\alpha}G$:
\begin{equation}
(h',g')\circ (h,g) =(hh', g)\qquad\hbox{if $\tau(h)g=g'$,}
\end{equation}
bearing in mind that $ t(h,g)=\tau(h)g$ and $s(h',g')=g'$. Multiplication is given by
\begin{equation}
(h_2,g_2)(h_1,g_1)=\bigl(h_2g_2h_1g_2^{-1}, g_2g_1\bigr).
\end{equation} 
Thus, recalling $\theta_{ik}(\gamma)$ from (\ref{E:tikjlmor}), we have
\begin{equation}
\theta_{ik}(\gamma)\theta_{km}(\gamma)=\bigl(h_{ik}(\gamma)g_{ik}(\gamma_0)h_{km}(\gamma)g_{ik}(\gamma_0)^{-1},\, g_{ik}(\gamma_0)g_{km}(\gamma_0)\bigr).
\end{equation}
and so
\begin{equation}
\begin{split}
&\mbt_{ikm}(v)\circ \bigl(\theta_{ik}(\gamma)\theta_{km}(\gamma)\bigr) \\
&=\bigl(h_{ikm}(v), g_{ik}(v)g_{km}(v)\bigr)\circ \bigl(h_{ik}(\gamma)g_{ik}(\gamma_0)h_{km}(\gamma)g_{ik}(\gamma_0)^{-1},  \\
&\hskip 3.5in   g_{ik}(\gamma_0)g_{km}(\gamma_0)\bigr)\\
&  =\bigl(h_{ikm}(v)h_{ik}(\gamma)g_{ik}(\gamma_0)h_{km}(\gamma)g_{ik}(\gamma_0)^{-1}, \, g_{ik}(\gamma_0)g_{km}(\gamma_0)\bigr).
\end{split}
\end{equation}
Focusing on the $H$-component we have:
\begin{equation}\label{E:hikmcomput}
\begin{split}
&
h_{ikm}(v)h_{ik}(\gamma)g_{ik}(\gamma_0)h_{km}(\gamma)g_{ik}(\gamma_0)^{-1} \\
&=h_{ikm}(v)h_{ik}(\gamma)h_{ik}(\gamma_0)h_{km}(\gamma)h_{ik}(\gamma_0)^{-1},
\end{split}
\end{equation}
upon using the second Peiffer identity (\ref{E:Peiffer}):
\begin{equation}
\tau(h)h_1\tau(h)^{-1}=hh_1h^{-1}\qquad\hbox{for all $h, h_1\in H$.}
\end{equation}
Continuing, we have
\begin{equation}\label{E:hikmcomput2}
\begin{split}
&h_{ikm}(v)h_{ik}(\gamma)g_{ik}(\gamma_0)h_{km}(\gamma)g_{ik}(\gamma_0)^{-1} \\
& = h_{ikm}(v)h_{ik}(\gamma_1)h_{ik}(\gamma_0)^{-1}\cdot h_{ik}(\gamma_0)\cdot h_{km}(\gamma_1)h_{km}(\gamma_0)^{-1}\cdot h_{ik}(\gamma_0)^{-1}\\
&= h_{ikm}(v)h_{ik}(\gamma_1)h_{km}(\gamma_1) h_{km}(\gamma_0)^{-1}  h_{ik}(\gamma_0)^{-1} \\
&=\Bigl( h_{ikm}(v)h_{ik}(v)h_{km}(v)\Bigr)\,\Bigl( h_{ikm}(u)h_{ik}(u)h_{km}(u)\Bigr)^{-1} h_{ikm}(u)\\
&= h_{im}(v)h_{im}(u)^{-1} h_{ikm}(u)
\end{split}
\end{equation}
where, in the last equality, we have used the first gerbal relation in (\ref{E:hijkgijik2}).   Thus
\begin{equation}\label{E:upperright}
\begin{split}
\mbt_{ikm}(v)\circ \bigl(\theta_{ik}(\gamma)\theta_{km}(\gamma)\bigr)&\\
&\hskip -1in =\bigl(h_{im}(v)h_{im}(u)^{-1} h_{ikm}(u), \, g_{ik}(\gamma_0)g_{km}(\gamma_0)\bigr).
\end{split}
\end{equation}

On the other hand,  $\theta_{im}(\gamma)\circ\mbt_{ikm}(u)$ is given by:
\begin{equation}\label{E:lowerright}
\begin{split}
\Theta_{im}(\gamma)\circ\mbt_{ikm}(u) &=\bigl(h_{im}(\gamma_1)h_{im}(\gamma_0)^{-1}h_{ikm}(u),\, g_{ik}(u)g_{km}(u)\bigr),
\end{split}
\end{equation}
on using the expression for $\mbt_{ikm}(u) $ from (\ref{E:htild}). 
Comparing with (\ref{E:upperright}) we conclude that the diagram (\ref{E:Tnatural}) commutes.
 \end{proof}
 Next let us make an observation about the product $\theta_{ik}(\gamma)\theta_{km}(\gamma)$:
 
 \begin{prop}\label{P:prodtheta} With notation as above, 
 \begin{equation}\label{E:prodtheta}
  \Theta_{ikm}(\gamma)
\theta_{ik}(\gamma)\theta_{km}(\gamma)=  \theta_{im}(\gamma)\end{equation}
 for all $\gamma\in\Mor(\mbu_{ikm})$, where 
 \begin{equation}\label{E:defTheta}
 \Theta_{ikm}(\gamma) =\bigl(h_{ikm}(\gamma), g_{ikm}(\gamma_0)\bigr),
 \end{equation}
 with 
 $$h_{ikm}(\gamma)=h_{ikm}(\gamma_1)h_{ikm}(\gamma_0)^{-1},$$
 where on the right we have the function $h_{ikm}$ on $U_{ikm}$ as in  (\ref{E:defgikhijk}), and $g_{ikm}(u)=\tau\bigl(h_{ikm}(u)\bigr)$ for all objects $u$ of $\mbu_{ikm}$. 
 \end{prop}
\begin{proof}  The $H$-component  of $\theta_{ik}(\gamma)\theta_{km}(\gamma)$ is
\begin{equation}\label{E:hikmstuff1}
\begin{split}
&\Bigl(\theta_{ik}(\gamma)\theta_{km}(\gamma)\Bigr)_H  \\
&=
 h_{ik}(\gamma)g_{ik}(\gamma_0)h_{km}(\gamma)g_{ik}(\gamma_0)^{-1} \\
&= h_{ik}(\gamma)h_{ik}(\gamma_0)h_{km}(\gamma)h_{ik}(\gamma_0)^{-1}\\
& = h_{ik}(\gamma_1)h_{ik}(\gamma_0)^{-1}\cdot h_{ik}(\gamma_0)\cdot h_{km}(\gamma_1)h_{km}(\gamma_0)^{-1}\cdot h_{ik}(\gamma_0)^{-1}\\
&= h_{ik}(\gamma_1)h_{km}(\gamma_1)\bigl(h_{ik}(\gamma_0)h_{km}(\gamma_0)\bigr)^{-1}.
\end{split}
\end{equation}
On the other hand, switching notation and writing $hg$ for  $(h,g)\in H\times_{\alpha}G$, and $g_{ikm}=\tau(h_{ikm})$, we have:
\begin{equation}
\begin{split}
 \Theta_{ikm}(\gamma)^{-1}\theta_{im}(\gamma) &= \bigl(h_{ikm}(\gamma_1)h_{ikm}(\gamma_0)^{-1} g_{ikm}(\gamma_0)\bigr)^{-1}\theta_{im}(\gamma)  \\
&= g_{ikm}(\gamma_0)^{-1}h_{ikm}(\gamma_0)h_{ikm}(\gamma_1)^{-1} \, h_{im}(\gamma_1) h_{im}(\gamma_0)^{-1}g_{im}(\gamma_0) \end{split}
\end{equation}
Inserting $g_{ikm}(\gamma_0)g_{ikm}(\gamma_0)^{-1}$ before the term $g_{im}(\gamma_0)$, we have:
\begin{equation}
\begin{split}
 & \Theta_{ikm}(\gamma)^{-1}\theta_{im}(\gamma) \\
&=g_{ikm}(\gamma_0)^{-1}h_{ikm}(\gamma_0)\cdot h_{ikm}(\gamma_1)^{-1}\,h_{im}(\gamma_1)\cdot h_{im}(\gamma_0)^{-1}g_{ikm}(\gamma_0)\cdot \\
&\qquad\qquad \cdot g_{ikm}(\gamma_0)^{-1}g_{im}(\gamma_0)\\ \\
&=h_{ikm}(\gamma_0)^{-1}h_{ikm}(\gamma_0)\cdot h_{ikm}(\gamma_1)^{-1}\,h_{im}(\gamma_1)\cdot h_{im}(\gamma_0)^{-1}h_{ikm}(\gamma_0)\cdot \\
&\qquad\qquad \cdot g_{ikm}(\gamma_0)^{-1}g_{im}(\gamma_0)
\end{split}
\end{equation}
where we have used  the second Peiffer identity (\ref{E:Peiffer}) to switch the conjugation by $g_{ikm}$ to conjugation by $h_{ikm}$. 
Next, using the gerbal relations (\ref{E:hijkgijik2}) we conclude that
\begin{equation}\label{E:hikmstuff2}
\begin{split}
 & \Theta_{ikm}(\gamma)^{-1}\theta_{im}(\gamma) \\
&=h_{ik}(\gamma_1)h_{km}(\gamma_1)\bigl(h_{ik}(\gamma_0)h_{km}(\gamma_0)\bigr)^{-1}  \cdot g_{ikm}(\gamma_0)^{-1}g_{im}(\gamma_0)\\
 \end{split}
\end{equation}
The $H$-component of this clearly matches the right hand side of (\ref{E:hikmstuff1}):
\begin{equation}\label{E:prodthetaH}
  \bigl(\Theta_{ikm}(\gamma)^{-1}\theta_{im}(\gamma)\bigr)_H=
\bigl(\theta_{ik}(\gamma)\theta_{km}(\gamma)\bigr)_H.\end{equation}

The $G$-component of $\theta_{ik}(\gamma)\theta_{km}(\gamma)$ is
$$g_{ik}(\gamma)g_{km}(\gamma)$$
which again matches the $G$-component $g_{ikm}(\gamma_0)^{-1}g_{im}(\gamma_0)$ on the right hand side of (\ref{E:hikmstuff2}) again by the  gerbal relations  (\ref{E:hijkgijik2}). 

We have thus shown that $ \Theta_{ikm}(\gamma)^{-1}\theta_{im}(\gamma)$ equals $
 \theta_{ik}(\gamma)\theta_{km}(\gamma)$.
   \end{proof}

   \subsection{Quotients for cocycles}
   
   We continue with the  same notation and framework.  In particular, $(G, H,\alpha,\tau)$ and $(H, J, \alpha',\tau')$ are Lie crossed modules. The image $\tau'(J)$ is a normal subgroup of $H$, as noted earlier in (\ref{E:normaltau}).   Then $\tau\tau'(J)$ is a subgroup of $G$.    We assume that $\tau'(J)$ and $\tau\tau'(J)$ are {\em closed} subgroups  of $H$ and of $G$, respectively. Let us observe that $\tau\tau'(J)$ is normal inside $\tau(H)$:
   \begin{equation}\label{E:normaltauH}
   \tau(h)\tau\tau'(j)\tau(h)^{-1} =\tau\bigl(h\tau'(j)h^{-1}\bigr)=\tau\tau'\bigl(\alpha'_h(j)\bigr)
   \end{equation}
   for all $h\in H$ and $j\in J$. Let $\ovG$ be the quotient group
   \begin{equation}\label{E:ovG}
   \ovG=G/\tau\tau'(J).
   \end{equation}  
   We recall  from (\ref{E:defgikhijk}) the functions
   $$g_{ik}=\tau (h_{ik}):U_{ik}\to G.$$
   Thus, by  the normality observation (\ref{E:normaltauH}), 
   $$g_{ik}(u)\tau\tau'(J)g_{ik}(u)^{-1}=\tau\tau'(J),$$
   for all $u\in U_{ik}$. Let us recall  from (\ref{E:hijkgijik2}) the gerbal relation
   \begin{equation}\label{E:gikgkm}
   g_{im}(u) =\tau\bigl(h_{ikm}(u)\bigr)g_{ik}(u)g_{km}(u)
   \end{equation} for all $u\in U_{ikm}$.  
 Thus the functions $g_{ik}$ form a cocycle modulo $\tau\tau'(J)$. Working with the  $\ovG$-valued   functions
    \begin{equation}\label{E:bartgik}
   \ovg_{ik}:U_{ik}\to\ovG:u\mapsto  {\overline g}_{ik}(u) =g_{ik}(u) \tau\tau'(J)\qquad\hbox{for $u\in U_{ikm}$,}
    \end{equation}
  we have then
    \begin{equation}\label{E:gikcocyc}
    \ovg_{ik}(u)\ovg_{km}(u) =\ovg_{im}(u)  \qquad\hbox{ for all $u\in U_{ikm}$.}  \end{equation}
    Thus, $\{\ovg_{ik}\}$ is a genuine $\ovG$-valued cocycle, associated to the covering $\{U_i\}_{i\in {\mathcal I}}$, in the traditional sense.
   
\subsection{A bundle from the cocycle}\label{ss:abfc}
   Let 
   $$X\to B$$
    be the  bundle  
     that is specified by the open covering $\{U_i\}_{i\in {\mathcal I}}$ and the $\ovG$-valued transition functions $\{\ovg_{ik}\}$.   A point of ${\mathbf X}$ is an equivalence class
     $$[i, u, \ovg],$$
     where $(i,u,\ovg), (j,v,\ovg')$, with $u\in U_i$, $v\in U_j$, and $\ovg,\ovg'\in\ovG$,  are `equivalent' if $v=u$ and
     $$\ovg'=\ovg_{ji}(u)\ovg.$$
     That this is in fact an equivalence relation on the set
     $$\cup_{i\in {\mathcal I}}\{i\}\times U_i\times \ovG$$
     is readily checked. We will return to more on this in section \ref{s:cbcld}.
     
     \subsection{Morphism cocycles} Turning now to morphisms, let us recall the relation (\ref{E:prodtheta}):
      \begin{equation}\label{E:prodtheta2}
  \Theta_{ikm}(\gamma)
\theta_{ik}(\gamma)\theta_{km}(\gamma)=  \theta_{im}(\gamma)\end{equation}
 for all $\gamma\in\Mor(\mbu_{ikm})$, where 
 \begin{equation}\label{E:defTheta2}
 \begin{split}
 \Theta_{ikm}(\gamma) &=\Bigl(h_{ikm}(\gamma), g_{ikm}(\gamma_0)\bigr) =\bigl(\tau'\bigl(j_{ikm}(\gamma)\bigr), \tau\bigl(h_{ikm}(\gamma_0)\bigr)\Bigr)\\
 &=\Bigl(\tau'\bigl(j_{ikm}(\gamma_1)j_{ikm}(\gamma_0)^{-1}\bigr), \, \tau\tau'\bigl(j_{ikm}(\gamma_0)\bigr)\Bigr),
 \end{split}
 \end{equation}
 and
 \begin{equation}\label{E:defthetaim}
 \theta_{im}(\gamma) =\bigl(h_{im}(\gamma), g_{im}(\gamma_0)\bigr).
 \end{equation}
 Thus the  $H\rtimes_{\alpha}G$-valued functions $\theta_{ij}$ form a cocycle modulo the elements of the form $\bigl(\tau'(j),\tau\tau'(j')\bigr)$ in $H\rtimes_{\alpha}G$, where $j, j'\in J$.  Let us denote the set of all such elements by $J_H$:
 \begin{equation}\label{E:defJH}
 J_H=\{\bigl(\tau'(j), \tau\tau'(j')\bigr)\,:\,j, j'\in J \}  \subset H\rtimes_{\alpha}G.
 \end{equation}
We now verify that $J_H$ is   a normal subgroup of $H\rtimes_{\alpha}\tau(H)$, so that there is a useful notion of equality `modulo' $J_H$.
 
 \begin{prop}\label{P:JH}  The mapping
 \begin{equation}\label{E:ovtau}
 {\overline\tau}:J\rtimes_{\alpha'}H\to H\rtimes_{\alpha}G: (j,h)\mapsto \bigl(\tau'(j), \tau(h)\bigr)
 \end{equation}
 is a homomorphism. The set $J_H$ forms a normal subgroup of $H\rtimes_{\alpha}\tau(H)\subset \Mor(\mbg)$.
 \end{prop}
 \begin{proof} First let us check that $\overline\tau$ is a homomorphism:
 \begin{equation}\label{E:JHsubgroup}
 \begin{split}
 \bigl(\tau'(j_2),\tau(h_2)\bigr) \bigl(\tau'(j_1),\tau(h_1)\bigr) &=
 \Bigl(\tau'(j_2)\alpha_{\tau(h_2)}\bigl(\tau'(j_1)\bigr), \,\tau(h_2)\tau(h_1)\Bigr)\\
 &=\Bigl(\tau'(j_2)h_2\tau'(j_1)h_2^{-1},\,\tau(h_2h_1)\Bigr)\\
 &\qquad\hbox{by the second Peiffer identity in (\ref{E:Peiffer}) }\\
&=\Bigl(\tau'(j_2) \tau'\bigl(\alpha'_{h_2}(j_1)\bigr),\,\tau(h_2h_1)\Bigr)\\
 &\qquad\hbox{by the first Peiffer identity in (\ref{E:Peiffer}) }\\
&=\Bigl(\tau'\bigl(j_2\alpha'_{h_2}(j_1)\bigr), \tau(h_2h_1)\Bigr). \end{split}
 \end{equation}
 Hence   $J_H$, being the image under ${\overline\tau}$  of the subgroup $J\rtimes_{\alpha'}\tau'(J)\subset J\rtimes_{\alpha'}H$,  is a subgroup of $H\rtimes_{\alpha}G$. This image is clearly contained inside  $H\rtimes_{\alpha}\tau(H)$.
 
 Let us now  work out how elements of  
 $${\rm Im}({\overline\tau})= \tau'(J)\rtimes_{\alpha}\tau(H)$$
  behave under conjugation by elements of $H\rtimes_{\alpha}\tau(H)$. The advantage of working with the smaller subgroup $H\rtimes_{\alpha}\tau(H)$ rather than $H\rtimes_{\alpha}G$ is that we have the relation
 $$\alpha_g\bigl(\tau'(j)\bigr)=h\tau'(j)h^{-1}=\tau'\bigl(\alpha'_h(j)\bigr)\in \tau'(J),$$
 if $g=\tau(h)$, $h\in H$ and $j\in J$. Now 
  let $h,h_0,h_1\in H$ and $j\in J$; then, working in the group $H\rtimes_{\alpha}G$, we have the conjugation
 \begin{equation}\label{E:JHnormal}\begin{split}
 \bigl(h,\tau(h_0)\bigr)\bigl(\tau'(j), \tau(h_1)\bigr)\bigl(h,\tau(h_0)\bigr)^{-1} &\\
 &\hskip -2in = \Bigl(h\alpha_{\tau(h_0)}\bigl(\tau'(j)\bigr), \,\tau(h_0h_1)\Bigr)\bigl(h,\tau(h_0)\bigr)^{-1} \\
 &\hskip -2in = \Bigl(h\alpha_{\tau(h_0)  }\bigl(\tau'(j)\bigr), \,\tau(h_0h_1)\Bigr)\Bigl(\alpha_{\tau(h_0) ^{-1} }(h^{-1}),  \tau(h_0)^{-1}\bigr)  \\
 &\hskip -2in =\Bigl(h\alpha_{\tau(h_0)}\bigl(\tau'(j)\bigr) \alpha_{\tau(h_0h_1)}\alpha_{\tau(h_0)^{-1}}(h^{-1}), \,\tau(h_0h_1h_0^{-1})\Bigr)\\
 &\hskip -2in= \Bigl(h\cdot h_0\tau'(j)h_0^{-1}\cdot (h_0h_1h_0^{-1})h^{-1}(h_0h_1h_0^{-1})^{-1}, \,\tau(h_0h_1h_0^{-1})\Bigr)\\
 &\hskip -2in=\Bigl(\tau'\bigl(\alpha'_{hh_0}(j)\bigr)\cdot hh_0h_1h_0^{-1}h^{-1}\cdot h_0h_1^{-1}h_0^{-1}, \,\tau(h_0h_1h_0^{-1})\Bigr).
 \end{split}
 \end{equation}
  As it stands it is not apparent if this is an element of $J_H$; however, we need to insert the condition that the second component of $\bigl(\tau'(j), \tau(h_1)\bigr)$ is in fact in $\tau\tau'(J)$, by requiring that
 $$ h_1 = \tau'(j_1),$$
 for some $j_1\in J$. Then we have
  \begin{equation}\label{E:JHnormal2}\begin{split}
 \bigl(h,\tau(h_0)\bigr)\bigl(\tau'(j), \tau(h_1)\bigr)\bigl(h,\tau(h_0)\bigr)^{-1} & \\
 &\hskip -2in=\Bigl(\tau'\bigl(\alpha'_{hh_0}(j)\bigr)\cdot hh_0h_1h_0^{-1}h^{-1}\cdot h_0h_1^{-1}h_0^{-1}, \,\tau(h_0h_1h_0^{-1})\Bigr)\\
 &\hskip -2in=\Bigl(\tau'\bigl(\alpha'_{hh_0}(j)\alpha'_{hh_0}(j_1)\alpha'_{h_0}({j_1}^{-1})\bigr),  \,\tau\tau'\bigl(\alpha'_{h_0}(j_1)\bigr) \Bigr),
 \end{split}
 \end{equation}
 which is indeed an element of $J_H$.
  \end{proof}

  \subsection{The quotient category ${\overline{\mbg}}$.}\label{ss:qG}
  We denote by $\ovGb$ the pair comprised of the  object set 
\begin{equation}\label{E:ovGb}
\Obj(\ovGb) =G/\tau\tau'(J)
\end{equation}
 and the morphism set
\begin{equation}\label{E:ovGbmor}
\Mor(\ovGb) =\bigl(H\rtimes_{\alpha}G\bigr)/J_H,
\end{equation}
where, as before,
$$J_H =\tau'(J)\times \tau\tau'(J)\subset H\rtimes_{\alpha}\tau(H)\subset H\rtimes_{\alpha}G=\Mor(\mbg).$$ 
Let us note that $J_H$ {\em is closed under composition}: if $f'_2=\bigl(\tau'(j_2), \tau\tau'(j'_2)\bigr)$ and $f'_1=\bigl(\tau'(j_1), \tau\tau'(j'_1)\bigr)$ are  such that the composition $f'_2\circ f'_1$ is defined as a morphism of $\mbg$ then
\begin{equation}\label{E:JHcomp}
f'_2\circ f'_1= \bigl(\tau'(j_2j_1), \tau\tau'(j_1)\bigr)\in J_H.
\end{equation}

We define source and target maps
\begin{equation}\label{stovGb}
\begin{split}
s: &\Mor(\ovGb)\to \Obj(\ovGb)  \\
t:& \Mor(\ovGb)\to \Obj(\ovGb) 
\end{split}
\end{equation}
to be the maps induced by the source and target maps in the category $\mbg$. The following result verifies that these maps are well-defined and lead to a category $\ovGb$.

\begin{prop}\label{P:ovGcat} $\ovGb$, as specified above, is a category under composition, source and target as inherited from $\mbg$.  The quotient mappings
\begin{equation}
\begin{split}
\Obj(\mbg)\to \Obj(\ovGb)&: g  \mapsto  g\tau\tau'(J) \\
\Mor(\mbg)\to \Mor(\ovGb)&: (h,g)  \mapsto  (h,g)J_H
\end{split}
\end{equation}
specify a functor
\begin{equation}\label{E:qfunctor}
q:\mbg\to\ovGb.
\end{equation}
\end{prop}

\begin{proof}
The source map
$$s:\Mor(\mbg)\to \Obj(\mbg)$$
maps  an element $\bigl(\tau'(j_1), \tau\tau'(j_2)\bigr)$  of the subgroup 
$$J_H =\tau'(J)\times \tau\tau'(J)\subset H\rtimes_{\alpha}\tau(H)\subset H\rtimes_{\alpha}G$$ 
to $\tau\tau'(j_2)$, which lies in $\tau\tau'(J)$. The target map 
$$t:\Mor(\mbg)\to \Obj(\mbg)$$
carries $\bigl(\tau'(j_1), \tau\tau'(j_2)\bigr)$ to $\tau\tau'(j_1j_2)$, which is also in $\tau\tau'(J)$. Hence $s$ and $t$ induce well-defined maps 
\begin{equation}
\bigl(H\rtimes_{\alpha}G\bigr)/J_H \to G/\tau\tau'(J)
\end{equation}
where the set on the left is the quotient set of cosets.

Let $f'_1, f'_2\in J_H$ be such that the composition
$$(f_2f'_2)\circ (f_1f'_1)$$
in $\Mor(\mbg)$ is meaningful. Then
\begin{equation}\label{E:f2fp2f1fp1}
(f_2f'_2)\circ (f_1f'_1) =(f_2\circ f_1)(f'_2\circ f'_1)\equiv (f_2\circ f_1)\qquad\hbox{mod $J_H$,}
\end{equation}
by (\ref{E:JHcomp}).  Hence the composition law in $\Mor(\mbg)$ induces a well-defined composition law in the quotient $\bigl(H\rtimes_{\alpha} G\bigr)/J_H$.

It is clear that for each $x=g\tau\tau'(J)\in \ovG$ the morphism $(e, g)\in H\rtimes_{\alpha}G$ induces the identity morphism $1_x$.

Thus $\ovGb$ is a category, and  the way we have defined source, target, identity morphisms, and composition in $\ovGb$ ensures that $q:\mbg\to\ovGb$ is a functor.
 \end{proof}

 \subsection{The categorical group $\ovGb_{\tau}$}\label{ss:qGt} Let us recall that $J_H$ is a normal subgroup inside $H\rtimes_{\alpha}\tau(H)$ rather than in $H\rtimes_{\alpha}G$. Thus the quotient
 $$(H\rtimes_{\alpha}\tau(H))/J_H$$
 is actually a group. Moreover, the subgroup  $H\rtimes_{\alpha}\tau(H)$ is closed under the composition law in $H\rtimes_{\alpha}G$; in fact if
 $$f_1=\bigl(h_1, \tau(h'_1)\bigr), f_2=\bigl(h_2, \tau(h'_2)\bigr)\in H\rtimes_{\alpha}\tau(H)\subset \Mor(\mbg),$$
are morphisms for which  the composition $f_2\circ f_1$ is meaningful, that is
$$\tau(h'_2)= \tau(h_1h'_1),$$
we have
\begin{equation}\label{E:compo}
f_2\circ f_1=\bigl(h_2h_1,\tau(h'_1)\bigr)\in H\rtimes_{\alpha}\tau(H).
\end{equation}
 
 As we have seen in (\ref{E:JHcomp}) $J_H$ is also closed under composition and by (\ref{E:f2fp2f1fp1}) the composition law in $H\rtimes_{\alpha}G$ then induces a well-defined operation on the quotient
 $$\bigl(H\rtimes_{\alpha}\tau(H)\bigr)/J_H.$$
 In summary, for the subcategory $\ovGb_{\tau}$ of $\ovGb$ specified by
 \begin{equation}\label{E:Gtaub}
 \begin{split}
 \Obj(\ovGb_{\tau}) &=\tau(H)/\tau\tau'(J)\\
 \Mor(\ovGb_{\tau}) &= \bigl(H\rtimes_{\alpha}\tau(H)\bigr)/J_H,
 \end{split}
 \end{equation}
 both object set and morphism set are groups, with the obvious quotient group structures. Moreover, since $s$ and $t$ are clearly group homomorphisms, $\ovGb_{\tau}$ is in fact a categorical group. Furthermore,  the `quotient' functor $q:\mbg\to\ovGb$ defined in (\ref{E:qfunctor}) restricts to a functor
 \begin{equation}\label{E:qfuncttau}
 q_{\tau}:\mbg_{\tau}\to \ovGb_{\tau},
 \end{equation}
  where $\mbg_{\tau}$ is the subcategory of $\mbg$ whose object set is $\tau(H)$ and whose morphisms are those morphism of $\mbg$ whose sources lie in $\tau(H)$. It is readily checked that $q_{\tau}$ 
  is a homomorphism both on objects and on morphisms.

 \subsection{Summary}  Starting with Lie crossed modules $(G, H, \alpha,\tau)$, associated with a categorical Lie group $\mbg$, and $(H, J,\alpha',\tau')$, associated with a categorical Lie group $\mbh$, and a manifold $B$, we introduced local data $(h_{ik}, j_{ikl})$,  named `gerbal cocycle', associated with an open covering $\{U_i\}_{i\in {\mathcal I}}$. From this data we  constructed  categories $\mbu_i$ (with objects being the points of $U_i$ and morphisms coming from the paths lying in $U_i$) and functors $\theta_{ik}:\mbu_{ik}\to \mbg$ that, along with certain natural transformations  ${\mathbb T}_{ikm}: \theta_{ik}\theta_{km}\Rightarrow\theta_{im}$, form a `functorial cocycle'.  We introduced a category $\ovGb$, by quotienting the object and morphism groups of $\Obj(\mbg)$ by subgroups specified by $\tau$ and $\tau'$. Inside $\ovGb$ is a subcategory $\ovGb_{\tau}$ that is a categorical group and contains all the geometrically relevant data. Our results in this section show that the system of function $\{\ovg_{ik}\}$, where $\ovg_{ik}=\tau(h_{ik})$, and $\{{\overline\theta}_{ik}\}$, where
 \begin{equation}\label{E:ovthetaik}
 {\overline\theta}_{ik}(\gamma)= { \theta}_{ik}(\gamma) J_H\in \bigl(H\rtimes_{\alpha}\tau(H)\bigr)/J_H,
 \end{equation}
 form cocycles in the traditional sense.

   \section{Construction of bundles from local data}\label{s:cbcld}
   
   In this section we start with a gerbal cocycle (\ref{E:hjgerb}) as described in the preceding section and construct a categorical bundle
   $${\mathbf X}\to {\mathbf B}$$
    as a limit of `local' categorical bundles ${\mathbf X}_{\alpha}\to {\mathbf U}_{\alpha}$. 
    
    We take as given a manifold $B$, an open covering $\{U_i\}_{i\in {\mathcal I}}$ of $B$, and a gerbal cocycle $\{h_{ik}, j_{ikl}\}$ as in (\ref{E:hjgerb}). We use the notation and constructions from the preceding section.  We also work with the  categorical group
    $$\ovGt,$$
    whose objects form  the quotient group
    $$\tau(H)/\tau\tau'(J),$$
    and whose morphisms form the quotient group
    $$\bigl(H\rtimes_{\alpha}\tau(H)\bigr)/J_H,$$
    where $J_H$ is the normal subgroup defined in (\ref{E:defJH}).
    We will make the standing assumption that $\tau(H)$, $\tau'(J)$ and $\tau\tau'(J)$ are {\em closed} subgroups; this ensures that the quotients are Lie groups.
    
    For the sake of practical convenience, mainly notational, we will assume that $\tau$ is surjective:
    \begin{equation}\label{E:tauH}
    \tau(H)=G.
    \end{equation}
    This makes it possible for us to work with  $\ovGb$ instead of $\ovGb_{\tau}$ as the structure categorical group for the resulting categorical principal bundle $\mbp$.

    \subsection{A  local  system of categorical bundles} As in subsection \ref{ss:covsubc}, we denote by ${S_{\mathcal I}}$ the set of all finite nonempty subsets  $I$ of ${\mathcal I}$ for which the intersection 
   \begin{equation}\label{E:Ualpha}
U_{I}=\cap_{k\in  I}U_k  \end{equation}
is nonempty. We can visualize $I$ itself as a simplex whose vertices are the points $i\in I$. We also have, for every $I\in {S_{\mathcal I}}$,  the category
$$\mbu_{I}$$
whose object set is $U_{I}$ and whose morphisms arise from the paths  on $B$ that lie entirely within each $U_i$ for $i\in I$:
\begin{equation}
\Mor(\mbu_{I})=\cap_{i\in I}\Mor(\mbu_i).
\end{equation}
Source, target, and composition are all inherited from any of the $\mbu_i$. It will be useful to keep track of the index $i$ and the collection of indices $I$.

We consider also the categories
$$\{(i,I)\}\times \mbu_I\times \ovGb,$$
where the objects are of the form $(i,I,u,\ovg)$, with $u\in U_I$ and $\ovg\in\ovG$,  and morphisms are of the form $(i,I, \gamma,\phi)$, where $\gamma\in\Mor(\mbu_I)$ and $\phi\in\Mor(\ovGb)$.
 Source and target maps are given by
  
  \begin{equation}\label{E:sfourtuple}\begin{split}
  s(i,I,\gamma,\phi)&= \bigl(i, I, s(\gamma), s(\phi)\bigr)\\
 t(i,I,\gamma,\phi)&= \bigl(i, I, t(\gamma), t(\phi)\bigr).
 \end{split}
 \end{equation}

 Consider a pair of indices $i,k\in {I}$, where $I\in {S_{\mathcal I}}$. We think of the open set $U_I$ and two trivializations of a bundle over $U_I$, with transition function $\ovg_{ik}$.  At the categorical level we have a functor
\begin{equation}\label{E:phiik}
\Phi_{ki}:\{(i,I)\}\times\mbu_{I}\times \overline{\mbg} \to \{(k,I)\}\times\mbu_{I}\times \overline{\mbg},
\end{equation}
which  is given on objects  by
$$(i,I, u, \ovg)\mapsto \bigl(k, I, u, {\ovg}_{ki}(u)\ovg\bigr)$$
and   is given on morphisms by
\begin{equation}(i,I, \gamma, \overline{\phi})\mapsto \bigl(k, I, \gamma, {\overline\theta_{ki}}(\gamma)\overline{\phi}\bigr).
\end{equation}
Moreover, if $I\subset J$ (so that $U_I$ is a larger set than $U_J$) and $i\in I$ then we have an `inclusion' morphism
\begin{equation}\label{E:incl}
\{(i,J)\}\times \mbu_J\times \ovGb\to \{(i,I)\}\times \mbu_I\times \ovGb
\end{equation}
which takes any object $(i,J,u,\ovg)$ to $(i,I,u,\ovg)$ and any morphism $(i,J, \gamma,\ovph)$ to $(i,I, \gamma,\ovph)$.

  \subsection{From local data to the global category} Our goal is to splice together the local data into a  base category and a `bundle' category ${\mathbf X}$. The base category is simply $\mbbb$, with object set the manifold $B$ and morphisms being smooth paths on $B$ with suitable identifications (as discussed in the context of (\ref{E:gam2gam1})). 
  We think of the category ${\mbx}$ as being a  limit of the system of categories $\{(i,I)\}\times \mbu_I\times{\ovl\mbg}$ along with the functors discussed in (\ref{E:phiik}) and (\ref{E:incl}). The  role played by $I$ in $(i,I,u,\ovg)$ is not essential and is meant only to help track the intersection sets $U_I$. 
  
\subsection{The object set of $\mbx$}  The object set of the category ${\mathbf X}$ is the bundle space $X$ for the bundle   obtained from the cocycles $\{{\overline g}_{ik}\,:\,i,k\in I\}$. Thus a point of $X$ is an equivalence class
  $$[i,u,\ovg],$$
  where $i\in I$, $u\in U_i$, $\ovg\in \ovG$, and the equivalence relation is defined by requiring that $[i,u,\ovg]$ and $[j,u',\ovg']$ be equal if and only if $u'=u\in U_i\cap U_j$ and
  $$\ovg'=g_{ji}(u)\ovg.$$
   We have then injections
 
 \begin{equation}
 \{(i,I)\}\times U_I\times \ovG\to X: (i,I,u,\ovg)\mapsto [i, u,\ovg].
 \end{equation}
 We have essentially seen this construction in subsection \ref{ss:abfc}.

\subsection{Towards morphisms}  Proceeding towards morphisms,   we need first the   quadruples
  $$(i, I, \gamma, \ovph),$$
  where  $i\in{\mathcal I}$, $I\in S_{\mathcal I}$, $\gamma\in \Mor({\mathbf U}_{I})$, $\ovph\in\Mor({\ovl\mbg})$.  Thus we have here a path $\gamma$ that lies inside the open set
  $$U_I=\cap_{j\in I}U_j,$$
  a particular index choice $i\in I$  that indicates a trivialization over $U_i$, and a morphism $\ovph =({\ovl h},\ovg)$ of ${\ovl\mbg}$ that we think of as indicating, through $\ovg$,  the `location' of the path in the bundle $X$  along with a decoration ${\ovl h}$ on it.  
     A difficulty arises when we try to form a composition
  $$(j, J, \gamma_j,\phi_j)\circ (i, I,\gamma_i,\phi_i)$$
  where the composition $s(\gamma_j)=t(\gamma_i)$. It is not apparent what this composition ought to be.  The `coarsest' solution to this difficulty is to define the composite simply to be the sequence of two quadruples
$$\bigl((i, I,\gamma_i,\phi_i), (j, J, \gamma_j,\phi_j)\bigr).$$

\subsection{Quivers}\label{ss:quiv} The language of quivers helps with the structure here. By a {\em quiver} $Q$  we mean a  set $E$ of  (directed) `edges' and a nonempty set $V$ of vertices, along with source and target maps $s, t:E\to V$. This gives rise to a category ${\mathbf C}_Q$, called the {\em free category} for $Q$: the object set is just $V$, and a morphism is a sequence $(e_1,\ldots, e_n)$ of edges with the target of each edge $e_i$ equal to the source of the next edge  $e_{i+1}$; in addition, we also include an identity morphism for each object. Composition is defined by concatenation of sequences. 

\subsection{The quiver of decorated local paths} Returning to our context, the vertex set $V$ for our quiver $Q$  is $X$, consisting of all equivalence classes $[i,I,u,\ovg]$; edges are    the quadruples
$$(i, I, \gamma_i, \phi_i),$$
with source and target given by
\begin{equation}\label{E:stquiver}\begin{split}
s(i, I, \gamma_i, \phi_i) &= [i, I, s(\gamma_i), s(\phi_i)]\\
t(i, I, \gamma_i, \phi_i) &= [i, I, t(\gamma_i), t(\phi_i)].
\end{split}
\end{equation}
  However, in the resulting quiver category ${\mathbf C}_Q$ certain morphisms need to be identified. It is convenient to do this at an algebraic level rather than `geometric', and so we turn first to the notion of the quiver algebra.
  
  \subsection{The quiver algebra}The {\em quiver algebra} $\mca_Q$ of a quiver $Q$  is the free algebra on the set of non-identity morphisms of the quiver category ${\mathbf C}_Q$ quotiented so that the product of two morphisms that have non-matching source-target relation is zero and in other cases the product is given by concatenation; thus an element of $\mca_Q$ can be expressed uniquely as a linear combination of sequences of paths of the form
$$(e_1,\ldots, e_n),$$
where $n\geq 1$, each $e_i\in E$,  and $t(e_i)=s(e_{i+1})$ for all $i\in\{1,\ldots,n\}$. The product of two such morphisms is given by
$$(e_1,\ldots, e_n)\cdot (e_{n+1},\ldots, e_m)= (e_1,\ldots, e_m) \quad\hbox{if $t(e_n)=s(e_{n+1})$,}$$
and is $0$ in all other cases. In particular, if $(e_1,\ldots, e_n)$ is above, with source-target matching for successive edges, then
\begin{equation}
   e_1e_2\ldots e_n=  (e_1,\ldots, e_n).
 \end{equation}

\subsection{A quotient of the quiver algebra} For the quiver algebra of our  path quiver ${\mathbf C}_Q$, let $N_0$ be the ideal in $\mca_{  Q }$ generated by all elements of the form
\begin{equation}\label{E:N0}
(i,I,\gamma, \ovph)- (j,J,\gamma, \ovt_{ij}(\gamma)\ovph),
\end{equation}
where $I, J\in S_{\mci}$, $i\in I$, $j\in J$, $\gamma$ is any morphism in $\mbu_{\{i,j\}}$, and $\ovph\in\Mor(\ovGb)$. Moreover, from a geometric viewpoint, if the sets $U_I$ and $U_J$ lie inside $U_K$ then, for any $\gamma_i\in\Mor(\mbu_i)$ and $\gamma_j\in\Mor(\mbu_j)$ for which $s(\gamma_j)=t(\gamma_i)$, the composition
  $$ ( j, J, \gamma_j,\ovph_j) \circ  (i, I,\gamma_i,\ovph_i)$$
  ought to be obtainable by transforming both to the $(k,K)$ system:
\begin{equation}
\begin{split}
&\bigl( k,K, \gamma_j,\ovt_{kj}(\gamma_j)\ovph_j)\bigr)\circ \bigl(k,K, \gamma_i, \ovt_{ki}(\gamma_i)\ovph_i\bigr)\\
&
=\Bigl(k, K, \gamma_j\circ\gamma_i, \bigl(\ovt_{kj}(\gamma_j)\ovph_j\bigr)\circ \bigl(\ovt_{ki}(\gamma_i)\ovph_i\bigr)\Bigr).
\end{split}
\end{equation}
With this in mind, we consider the ideal $N$   generated by $N_0$ along with all elements of  $\mca_{ Q }$ of the form

\begin{equation}\label{E:compiden}
(k,K, \gamma_k,  \ovph_k) -  (i, I,\gamma_i,\ovph_i)\cdot  ( j, J, \gamma_j,\ovph_j),
\end{equation}
where $I, J, K\in S_\mci$ are such that
 \begin{equation}\label{E:equivmor1}
  K\subset I\cap J \qquad\hbox{(which implies $U_I, U_J\subset U_K$) }
   \end{equation}
   and
  \begin{equation}\label{E:equivmor}
  \begin{split}
  \gamma_j\circ\gamma_i &=\gamma_k\\
 \bigl( \ovt_{kj}(\gamma_j)\ovph_j\bigr)\circ \bigl( \ovt_{ki}(\gamma_i)\ovph_i\bigr) &=\ovph_k.
 \end{split}
 \end{equation}
 Thus in the quotient algebra $\mca_Q/N$ we have
 \begin{equation}\label{E:mcaQprod}
 [k, K, \gamma_k, \ovph_k] =[i, I,\gamma_i,\ovph_i][ j, J, \gamma_j,\ovph_j].
 \end{equation}

 \subsection{Equivalent morphisms in the quotient} We define two non-identity morphisms of the category ${\mathbf C}_Q$ to be equivalent if they have the same image in the quotient algebra ${\mathcal A}_Q/N$.  We denote by
 $$[i, I, \gamma, \ovph]$$
 the equivalence class in ${\mathcal A}_Q/N$ corresponding to the element $(i,I, \gamma,\ovph)$ in ${\mathcal A}_Q$. Thus,
 \begin{equation}\label{E:Morident1}
[i,I,\gamma, \ovph] =[j,J,\gamma, \ovt_{ij}(\gamma)\ovph]
\end{equation}
if $\gamma\in\Mor(\mbu_{\{i,j\}})$.

   We denote the set of equivalence classes of non-identity morphisms of $\mbbc_Q$ by
\begin{equation}\label{E:mormbx}
\Mor(\mbx)_0,
\end{equation}
with notation anticipating our objective of showing that $\mbx$ is a category. The subscript $0$ here is a `temporary hold' pending inclusion of identity morphisms. 

\subsection{The mor-set $\Mor(\mbx)$} Now we include an identity morphism $1_x$ for each object $x\in \Obj(\mbx)$, and define
\begin{equation}
\Mor(\mbx)=\{1_x\,:\,x\in\Obj(\mbx)\}\cup \Mor(\mbx)_0.
\end{equation}

 \subsection{Source and target} The source and target maps
 $$s, t:E\to V$$
 of a quiver $Q$
 give rise to linear maps
 \begin{equation}
 s, t: \mca_Q  \to \mbr[V],
 \end{equation}
 where the latter is the free real vector space over the set $V$ of vertices (thus every element of $\mbr[V]$ is uniquely a  (finite) linear combination of the form $\sum_ja_jv_j$, with $a_j\in \mbr$ and $v_j\in V$). These maps do not respect multiplication since
 $$s(e_1e_2)=s(e_1) \qquad\hbox{and}\qquad t(e_1e_2)=t(e_2),$$
 for example. Now returning to our context, let us note that, with notation as in (\ref{E:compiden}), 
   \begin{equation}
 \begin{split}
 s(\ovph_k)= s\bigl( \ovt_{ki}(\gamma_i)\ovph_i\bigr)&=s\bigl(\ovt_{ki}(\gamma_i)\bigr)s(\ovph_i)\\
 &=\ovt_{ki}\bigl(s(\gamma_i)\bigr)s(\ovph_i),
 \end{split}
 \end{equation}
 and similarly for the targets.  Consequently we have the following source/target consistency:
 \begin{equation}\label{E:stconsist}
 \begin{split}
 s(k,K, s(\gamma_k),\ovph_k)  &=\bigl(k, K, s(\gamma_i),\ovt_{ki}\bigl(s(\gamma_i)\bigr)s(\ovph_i)\bigr) =s(i, I, s(\gamma_i),\ovph_i);\\
 t(k,K,\gamma_k,\ovph_k) &=t(j,  J, t(\gamma_j), \ovph_j).
 \end{split}
 \end{equation}
 Moreover, for $\gamma\in\mbu_{\{i,j\}}$, where $i\in I$ and $j\in J$, and any $\ovph\in\Mor(\ovGb)$, we have
\begin{equation}\label{E:socinsist2}
\begin{split}
s(i,I,\gamma,\ovph)&=s\bigl(j, J, s(\gamma), \ovt_{ij}(\gamma)\ovph\bigr)\\
t(i,I,\gamma,\ovph)&=t\bigl(j, J, t(\gamma), \ovt_{ij}(\gamma)\ovph\bigr).
\end{split}
\end{equation}
Because of these relations and (\ref{E:stconsist}),   the value of $s$ on a quiver algebra element of the form $e_1\ldots e_n$, being just $s(e_1)$,  remains unchanged if any of the $e_i$ is replaced by either a  single-edge element $e$ or a two-edge morphism $ee'$ related to $e_i$ in any of the ways described above in (\ref{E:Morident1}). An analogous statement holds for the target map $t$. Hence the source and target maps descend to well-defined linear maps
\begin{equation}\label{E:stquot}
s, t:\mca_{Q}/N\to \mbr[V].
\end{equation}
In particular, $s$ and $t$ are well-defined maps when restricted to $\Mor(\mbx)_0$ (the values of $s$ and $t$ on elements of $\Mor(\mbx)_0$ are elements of $V$). For the identity morphism $1_x$ we define both source and target to be $x$.  

\subsection{Composition of morphisms in $\mbx$} The multiplication operation on $\mca_Q$ induces a well-defined operation on the quotient $\mca_Q/N$ and this restricts to a well-defined operation on $\Mor(\mbx)_0$. We use this to define composition of elements in $\Mor(\mbx)_0$. Concretely, composition is given  by concatenation:
\begin{equation}\label{E:MorXcompo}
(e_{n+1}\ldots e_{n+m})\circ (e_1\ldots e_n)= e_1\ldots e_n e_{n+1}\ldots e_{n+m},
\end{equation}
where each $e_j$ is of the form
$$e_j =[j, J, \gamma_j, \ovph_j]$$
and
$$t(e_n)=s(e_{n+1}).$$
(This last condition ensures that the product on the right in (\ref{E:MorXcompo}) is not $0$ but   an element of $\Mor(\mbx)_0$ instead.)
The significance of (\ref{E:MorXcompo}) is that the value of the right hand side as an element of $\Mor(\mbx)_0$ is independent of the specific choices of the $e_j$ used in representing the elements of $\Mor(\mbx)$ being composed on the left. The definition (\ref{E:MorXcompo}) makes it clear that if $f, g\in\Mor(\mbx)$ for which $t(f)=s(g)$ then $g\circ f$ is defined an
\begin{equation}
s(g\circ f)=s(f)\qquad\hbox{and}\qquad t(g\circ f)=t(g).
\end{equation}
Of course, we define composition with identity morphisms by
\begin{equation}
f\circ 1_x=f \qquad\hbox{and}\qquad 1_y\circ f=f
\end{equation}
if $x=s(f)$ and $y=t(f)$.
Associativity of the composition law on $\Mor(\mbx)$ follows from associativity of the concatenation process.

  We have thus constructed the category $\mbx$.
  
  \subsection{Projection to $\mbbb$} There is a well-defined projection functor 
  $$\pi: \mbx\to \mbbb$$
  given by
  \begin{equation}\label{E:XBproj}
  \begin{split}
  [i,I,u,\ovg] &\mapsto u\\
  \bigl[i_1, I_1, \gamma_{i_1},\phi_{i_1}]\ldots \bigl[i_n, I_n, \gamma_{i_n},\phi_{i_n}\bigr] &\mapsto \gamma_{i_n}\circ\ldots\gamma_{i_1}
  \end{split}
  \end{equation}
  That this is well-defined on morphisms follows from the first relation in (\ref{E:equivmor}).
  
   \subsection{Local triviality}   It is clear that at the object-level we have  traditional local triviality of the principal $\ovG$-bundle $X\to B$. At the level of morphisms any morphism of $\mbx$ that projects to a morphism $\gamma$ that lies entirely inside $U_I$ arises from some
  $$(i,I, \gamma, \phi)\in \Mor\bigl(\{(i,I)\}\times \mbu_I\times {\ovl{\mbg}}\bigr).$$
  Thus the category $\mbx_I$, with objects in $\pi^{-1}(U_I)$ and morphism projecting to $\mbu_I$, is isomorphic to $\mbu_I\times\ovGb$:
 \begin{equation}\label{E:XBIproj}
\Phi_{i,I}:\mbu_I\times\ovGb\to \mbx_I: \begin{cases} (u,\ovg)&\mapsto [i,I,u, \ovg];\\
(\gamma,\ovph) &\mapsto [i, I, \gamma, \ovph].
\end{cases}
  \end{equation}
  \subsection{Action of the group $\ovGb$} There is a well-defined right action
  \begin{equation}
  \mbx\times\ovGb\to \mbx: \begin{cases}\bigl( [i,I,u], \ovg\bigr) &\mapsto  [i, I,u\ovg];\\
  \bigl([i,I,\gamma,{\overline\phi}], {\overline\psi}\bigr) &\mapsto [i, I, \gamma, {\overline\phi}{\overline\psi}].\end{cases}
  \end{equation}
  That the action is well-defined on morphisms may be seen from
  \begin{equation}
  ({\overline\phi}_2\circ{\overline\phi}_1){\overline\psi}= ({\overline\phi}_2{\overline\psi})\circ ({\overline\phi}_1{\overline\psi}).
  \end{equation}
  Using local triviality it follows that this action is free. 
  
We have thus shown that 
$$\mbx\to\mbbb$$
 is a categorical principal $\ovGb$-bundle. Let us recall that we assumed in (\ref{E:tauH}) that $\tau(H)=G$. Dropping this assumption means simply that $\mbx\to\mbbb$ is a categorical principal $\ovGb_{\tau}$-bundle.

        {\bf{ Acknowledgments.} }   Sengupta   acknowledges    research support from   NSA grant H98230-13-1-0210,  and    the SN Bose National Centre for its hospitality during visits when part of this work was done.

\end{document}